\documentclass[a4paper, 12pt, twoside, draft]{article}

\usepackage{graphicx}
\usepackage{amssymb,amsfonts,amsthm, amscd}
\usepackage{amsmath}
\usepackage{hyperref}
\usepackage[mathcal]{eucal}

\newtheorem{thm}{{\sc Theorem}}[section]

\newtheorem{lem}{{\sc Lemma}}[section]
\newtheorem{cor}{{\sc Corollary}}[section]

\begin{document}

\author{V.\,S.~Atabekyan}
\title{Splitting automorphisms of prime power orders of free Burnside groups}


\footnotetext{The author was supported in part by State Committee Science MES RA grant in frame of project 13-1A246}

\maketitle

\begin{abstract}
We prove that if the order of a splitting automorphism of free Burnside group~$B(m,n)$ of odd period~$n\ge1003$ is a prime power, then the automorphism is inner. Thus, we give an affirmative answer to the question on the coincidence of splitting and inner automorphisms of free Burnside groups $B(m,n)$ for automorphisms of orders~$p^k$ ($p$ is a prime number). This question was posed in the Kourovka Notebook in 1990 (see 11th ed., Question 11.36.~b).

\end{abstract}



\section{Introduction}\label{s1}

An automorphism $\varphi$ of a group $G$ is called \textit{a splitting auto\-mor\-phism of period  $n$}, if
$\varphi^n=1$ and $g\,g^{\varphi}g^{\varphi^2}\cdot\cdot\cdot
g^{\varphi^{n-1}}=1$ for all $g\in G$. Various authors studied groups with splitting automorphism. The known theorem of O.Kegel, proved in 1961, states that any finite group possessing a nontrivial splitting automorphism of prime order is nilpotent (see \cite{Keg}).
E.Khukhro proved that any solvable group possessing a nontrivial splitting automorphism of prime order is also nilpotent (see \cite{Khu}). E.Jabara in \cite{J} established that a finite group with the splitting automorphism of order 4 is solvable.

It is not difficult to check that if $\varphi$ is a splitting automorphism of period $n$ of $G$, then $g^{\varphi^{n-1}}\cdot\cdot\cdot g^{\varphi^2}g^{\varphi}\,g=1$ for any $g\in G$. Let $\varphi$ be a splitting automorphism of period $n$ of some group $G$. Since $(\varphi\, g)^n=\varphi^ng^{\varphi^{n-1}}\cdot\cdot\cdot
g^{\varphi^2} g^{\varphi}\,g$, this precisely means that the relation $(\varphi\, g)^n=1$ holds in the holomorph $Hol(G)$ of $G$ for each $g\in G$. In particular, the identity automorphism of a given group $G$ is a splitting automorphism of period $n$ if and only if the identity
$x^n=1$ holds in $G$.

It is easy to check that if in a given group $G$ the identity $x^n=1$ is valid, then each inner automorphism of $G$ is an inner splitting automorphism of period $n$. However, the converse statement is false. Indeed, let $F_2$ be an absolutely free group with free generators $a$
and $b$ and let $C_n$ be a cyclic group of order $n$ with generator element $c$. Let us
consider the group $\Gamma = F_2/F_2^n\times C_n$ of period $n$, where $F_2^n$ stands for the subgroup
generated by all possible $n$th powers of the elements of $F_2$. It can readily be seen
that for any $n > 1$, the automorphism $\alpha: \Gamma \to \Gamma$ given on the generators by
the formulae $\alpha(a) = ac$, $\alpha(b) = bc$, and $\alpha(c) = c$, is a splitting automorphism of period $n$.
However, this is an outer automorphism because it is clear that no relation of the
form $u^{-1}au = ac$ can hold in $\Gamma$.

\medskip
S.V.Ivanov in \cite{KT} posed the following problem: \textit{Let $n$ be large enough odd number and $m>1$. Is it true that each automorphism $\varphi$ of $B(m,n)$, that satisfies the relations $\varphi^n=1$ and
$g\,g^{\varphi}g^{\varphi^2}\cdot\cdot\cdot g^{\varphi^{n-1}}=1$ for any $g$ in $ B(m,n)$, is inner} (see \cite{KT}, Question 11.36. b))? In fact, Ivanov's problem concerns splitting automorphisms of period $n$ of the groups $B(m,n)$. By definition, the free Burnside group $B(m,n)$ of period $n$ and rank $m$ has the following presentation
$$B(m,n)=\langle a_1, a_2, ..., a_m \mid X^n=1\rangle ,
$$
where $X$ ranges over the set of all words in the alphabet $\{a_1^{\pm1},a_2^{\pm1},\ldots,a_m^{\pm1}\}$. The group $B(m,n)$ is a quotient group of the free group $F_m$ of rank $m$ by normal subgroup $F_{m}^n$, generated by all $n$th powers of the elements of $F_m$. Every periodic group of period $n$ with $m$ generators is a quotient group of $B(m,n)$.

The main Theorem \ref{t1} of the present paper gives a positive answer to the above mentioned question for all splitting automorphisms of prime power order of free Burnside groups $B(m,n)$ for odd periods $n\ge1003$. In the paper \cite{A13}, we established some properties of splitting automorphisms of the groups $B(m,n)$. Those properties have proved useful in our proof of the main result \ref{t1}. Theorem \ref{t1} strengthens the result of \cite{A13}, where the Ivanov's problem was solved only for automorphisms of prime orders.

\subsection{On simple periodic groups}\label{s2}

S.I.Adian and I.G.Lysenok in \cite{AL} proved, that for any $m>1$ and odd $n\ge 1003$ there exists a maximal normal subgroup $N$ of free Burnside group $B(m,n)$ such that the quotient group $B(m,n)/N$ is an infinite group, every proper subgroup of which is contained in some cyclic subgroup of order $n$. The groups constructed in \cite{AL} are infinite simple groups in which the identity relation $x^n=1$ holds. This groups are called 'Tarski monsters' since  Tarski has formulated a question on the existence of such groups. The first examples of Tarski monsters were constructed by A.Yu.Olshanskii \cite{Olmon} for prime periods $n>10^{75}$. It is now known that for every odd $n\ge 1003$ there are continuum many
non-isomorphic Tarski-monsters of period $n$ (see \cite{At87}, \cite[Theorem 28.7]{O89}, \cite{At07}). Each of these Tarski monsters is a result of a factorization of the group $B(m,n)$ by some maximal normal subgroup of, which leads to an infinite group containing only cyclic proper subgroups.

We denote by $\mathcal{M}_n$ the set of all normal subgroups $N\lhd B(m,n)$ for which the quotient group $B(m,n)/N$ is a Tarski monster.

The following two statements were proved in \cite{At11} (see also \cite{Cher}) and \cite{A13} respectively. They play the key role in the proof of main Theorem \ref{t1}.

\begin{lem}{\rm(see \cite[Corrolary 2]{At11})}\label{pr2}
Let $n\ge1003$ be an odd number and let $\varphi$ be an automorphism of $B(m,n)$ such that for any normal subgroup
$N\in\mathcal{M}_n$ the equality $N^{\varphi}=N$ holds. Then
$\varphi$ is an inner automorphism.
\end{lem}

\begin{lem}{\rm(see \cite[Lemma 4]{A13})}\label{triv}
If $\varphi$ is an arbitrary nontrivial splitting automorphism of period $n$ of the group $B(m,n)$, where $n\ge100 3$ is odd, then the stabilizer of any normal subgroup $N\in \mathcal{M}_n$ under the action of the cyclic group $\langle\varphi\rangle$ is nontrivial.
\end{lem}

\section{The main result}\label{s5}
\begin{thm}
\label{t1}
Let $\phi$ be a splitting automorphism of period $n$ of the group $B(m,n)$, where $n\ge1003$ is an odd number. If the order of automorphism $\phi$ is a power of some prime number, then $\phi$ is an inner automorphism.
\end{thm}
\begin{proof} According to Lemma \ref{pr2}, to prove the Theorem \ref{s5} it suffices to show that the equality $N^{\phi}=N$ holds for any normal subgroup $N\in\mathcal{M}_n$.

Suppose that there is a normal subgroup $A\in\mathcal{M}_n$ which is not $\phi$-invariant, that is $A^{\phi}\not=A$. On the other hand, by Lemma \ref{triv} the centralizer of each such subgroup $A$ is not trivial. Let $p^r$ be the order of automorphism $\phi$, where $p$ is some prime number. The number $p^r$ divides $n$ by definition.

Since the subgroups of the cyclic group $\langle\phi\rangle$ of order $p^r$ are linearly ordered by inclusion, one can choose some subgroup $N$ with the minimal centralizer among all non-$\phi$-invariant subgroups $A\in\mathcal{M}_n$. Being a subgroup of the group $\langle\phi\rangle$ of order $p^r$, this minimal nontrivial centralizer  is generated by some automorphism of the form $\phi^{p^k}$, where $1<k<r$.

By virtue of the minimality, the subgroup $\langle\phi^{p^k}\rangle$ is contained in the centralizer of any subgroup $A\in\mathcal{M}_n$. Hence, the automorphism $\phi^{p^k}$ centralizes all subgroups $A\in\mathcal{M}_n$. According to Lemma \ref{pr2} we obtain that automorphism $\phi^{p^k}$ is inner.

We shall use the following lemma, that was proved in the paper \cite{A13}.

\begin{lem}{\rm(see \cite[Lemma 3]{A13})}\label{R}
Let $\varphi:G\to G$ be an arbitrary automorphism and $H$ be a normal subgroup of the group $G$ such that the quotient group $G/H$ is a non-abelian and simple. In that case, if the subgroups $H,\,H^{\varphi}\,,...,H^{\varphi^{k-1}}$ are pairwise distinct and $H^{\varphi^{k}} =H$, then the quotient group
$G/\mathop\cap\limits_{i=1}^kH^{\varphi^{i}}$ is decomposed into the direct product of normal subgroups $H_j/\mathop\cap\limits_{i=1}^kH^{\varphi^{i}}$, $j=1,2,...,k$, wherein each quotient group $H_j/\mathop\cap\limits_{i=1}^kH^{\varphi^{i}}$ is isomorphic to $G/H$ and $H_j=\mathop\cap\limits_{\mathop{i=1}\limits_{i\not=j}}^kH^{\varphi^{i}}$.
\end{lem}

To use Lemma \ref{R}, suppose that $G=B(m,n)$,
$\varphi=\phi$ and let $H=N$ be the normal subgroup of the group $B(m,n)$ with the minimal normalizer $\langle\phi^{p^k}\rangle$. The quotient group $B(m,n)/N$ is a Tarski monster, since $N\in\mathcal{M}_n$. In particular, $B(m,n)/N$ is a non-abelian and simple group. By Lemma \ref{R}, the quotient group $B(m,n)/K$ is decomposed into the direct product of subgroups $N_0/K$, $N_1/K$, ... , $N_{p^k-1}/K$, where $K=\mathop\cap\limits_{i=0}^{p^k-1}N^{\phi^{i}}$.

As it was mentioned above, the automorphism $\phi^{p^k}$ is inner, that is for some element $u\in B(m,n)$ we have the equality $\phi^{p^k}=i_u$, where $i_u$ is the inner automorphism generated by $u$. Since the automorphism $\phi^{p^k}$ has order $p^{r-k}$, the relation $(i_u)^{p^{r-k}}(x)=x$ holds for any $x\in B(m,n)$. In the other words the element $u^{p^{r-k}}$ belongs to the center of $B(m,n)$. From Adian's theorem on the triviality of the center of the group $B(m,n)$ (see \cite[Theorem 3.4]{A}) it follows the equality $u^{p^{r-k}}=1.$ Moreover, the element $u$ has order $p^{r-k}$ because the automorphism $\phi^{p^k}$ has order $p^{r-k}$. 

For the element $uK$
of the group $B(m,n)/K$ there exist uniquely defined elements $u_0K,$ $u_1K,..., u_{p^k-1}K$ from the subgroups $N_0/K$, $N_1/K$, ... , $N_{p^k-1}/K$ respectively such that
\begin{equation}\label{uK}uK=u_0K\cdot u_1K\cdot\cdot\cdot u_{p^k-1}K.\end{equation}

Note that the relation $u^{p^{r-k}}=1$  implies immediately the relations $u_i^{p^{r-k}}K=K$ for all $i=0, 1,...,p^k-1.$ In particular, we have $u_0^{p^{r-k}}K=K.$

According to Lemma \ref{R} the groups $N_0/K$ and $B(m,n)/N$ are isomorphic. By definition of the set $\mathcal{M}_n$ any element of the quotient group $B(m,n)/N$ is contained in some cyclic subgroup of order $n$ (see \cite[Proposition 5.2]{AL}). Therefore, there exists an element $a\in B(m,n)$ such that $ aK\in N_0/K$ and the element $au_0K$ has the order $n$.

Besides, the elements $u_1K,\,\cdot\cdot\cdot, u_{p^k-1}K$ commute with both of the elements $aK$ and $u_0K$ in $B(m,n)/K$ since $u_1K,\,\cdot\cdot\cdot, u_{p^k-1}K$ belong to the direct factors $N_1/K$, ... , $N_{p^k-1}/K$ respectively. Hence, the equality \eqref{uK} implies the relations
\begin{equation}\label{u0} u^sau^{-s}K=u_0^sau_0^{-s}K
\end{equation}
for all integers $s$. Since $\phi$ is a splitting automorphism, we have the equality $a\,a^{\phi}a^{\phi^2}\cdot\cdot\cdot a^{\phi^{n-1}}=1$.

Taking into account $\phi$-invariance of the subgroup $K$, we obtain that the equality
\begin{equation}\label{aK}aK\cdot a^{\phi}K\cdot
a^{\phi^2}K\cdot\cdot\cdot a^{\phi^{n-1}}K=K\end{equation} holds in the quotient group $B(m,n)/K$.

Recall that according to choice of subgroup $N$ we have $a^{\phi^{p^k}}K\in N_0/K$. This allows us to rewrite the equality
\eqref{aK} in the form
\begin{equation}
\label{aK1}
bK\cdot b^{\phi}K\cdot b^{\phi^2}K\cdot\cdot\cdot b^{\phi^{p^k-1}}K=K,
\end{equation}
where
\begin{equation}\label{b}
b=a a^{\phi^{p^k}} a^{\phi^{2p^k}}\cdot\cdot\cdot
a^{\phi^{(n/p^{k}-1)p^k}}.
\end{equation}
Moreover, it is easy to see that the elements $bK,\, b^{\phi}K,\, b^{\phi^2}K,\cdot\cdot\cdot,\,
b^{\phi^{p^k-1}}K$ belong to the direct components $N_0/K$, $N_1/K$,
... , $N_{p^k-1}/K$ respectively. Hence, the equality \eqref{aK1} immediately implies that all the factors on the left-hand side of the equality \eqref{aK1} are trivial. In particular, we have the equality $bK=K$.

Since $\phi^{p^k}=i_u$, one can rewrite the equality \eqref{b} in the form
$$b=a\cdot uau^{-1}\cdot u^2au^{-2}\cdot
u^{n/p^{k}-1}au^{-(n/p^{k}-1)}.
$$
Accordingly, we obtain equality
$$bK=aK\cdot uau^{-1}K\cdot u^2au^{-2}K\cdot\cdot\cdot
u^{n/p^{k}-1}au^{-(n/p^{k}-1)}K$$ in the quotient group. Then using the relations \eqref{u0}, we get the equality
$$ bK=a\cdot u_0au_0^{-1}\cdot u_0^2au_0^{-2}\cdot\cdot\cdot
u_0^{n/p^{k}-1}au_0^{-(n/p^{k}-1)}K.$$

Next, we use the following identity
$$ a\cdot u_0au_0^{-1}\cdot
u_0^2au_0^{-2}\cdot\cdot\cdot
u_0^{n/p^{k}-1}au_0^{-(n/p^{k}-1)}=(au_0)^{n/p^{k}}\cdot
u_0^{-n/p^{k}}.$$
The number $n/p^{k}$ is divided by $p^{r-k}$  since $n$ is divided by $p^r$. By virtue of the equality $u_0^{p^{r-k}}K=K,$
we finally obtain $bK=(au_0)^{n/p^{k}}K$. Hence, we obtain
$(au_0)^{n/p^{k}}K=K$ because $bK=K$.  This contradicts to condition of choice of the the element $a$, according to which the element $au_0K$ has order $n$ in the group $B(m,n)/K$. Theorem is proved.

\begin{cor}
\label{t1}
If $p$ is an odd prime number and $n=p^k\ge1003$, then any splitting automorphism of period $n$ of the group $B(m,n)$ is an inner automorphism.
\end{cor}

\end{proof}

\end{document}